\documentclass[11pt]{amsart}
\usepackage{amssymb}
\usepackage{amscd}
\usepackage{color}
\usepackage{comment}
\usepackage[dvipsnames]{xcolor}
\usepackage{url}
\usepackage{hyperref}
\usepackage{enumitem}

\newtheorem{theorem}{Theorem}[section]
\newtheorem{lemma}[theorem]{Lemma}
\newtheorem{proposition}[theorem]{Proposition}
\newtheorem{corollary}[theorem]{Corollary}

\theoremstyle{definition}

  {\par \hfill \fbox{}}

\newtheorem*{theorem*}{Theorem}

\theoremstyle{remark}
\newtheorem{remark}[theorem]{Remark}

\numberwithin{equation}{section}

\newcommand{\X}{X}

\newcommand{\EL}{\mathcal{L}}

\newcommand{\C}{\mathbb{C}}

\newcommand{\D}{\mathbb{D}}
\newcommand{\T}{\mathbb{T}}

\newcommand{\RR}{\mathbb{{R}}}
\newcommand{\ZZ}{\mathbb{{Z}}}
\newcommand{\NN}{\mathbb{{N}}}

\newcommand{\CC}{\mathbb{{C}}}
\newcommand{\DD}{\mathbb{{D}}}

\newcommand{\R}{\mathbb{R}}

\newcommand{\B}{\mathcal{B}}

\newcommand{\Lat}{\text{Lat} \ }

\begin{document}

\title[Universal operators]{Universality arising from invertible weighted \\ composition operators}

\author[L. Abadias]{Luciano Abadias}
	\address[L. Abad\'ias]{Departamento de Matem\'aticas, Instituto Universitario de Matem\'aticas y Aplicaciones, Universidad de Zaragoza, 50009 Zaragoza, Spain.}
	\email{labadias@unizar.es}

\author[F. J. González-Doña]{F. Javier González-Doña}
	\address[F. J. González-Doña]{Departamento de Matemáticas, Universidad Carlos III de Madrid,  Avenida de la Universidad 30, 28911 Leganés, Madrid, Spain \and Instituto de Ciencias Matematicas ICMAT (CSIC-UAM-UC3M-UCM), Madrid, Spain}
	\email{fragonza@math.uc3m.es}
	
\author[J. Oliva-Maza]{Jes\'us Oliva-Maza}
	\address[J. Oliva-Maza]{Institute of Mathematics, Polish Academy of Sciences. Chopin Street 12/18, 87-100 Toruń, Poland}
	\email{joliva@impan.pl}
	
\thanks{Authors  have been partially supported by Project PID2022-137294NB-I00, DGI-FEDER, of the MCYTS. First and third author have been partially supported by Project E48\_23R, D.G. Arag\'on, Universidad de Zaragoza, Spain. Third author has also been partially supported by XXXV Grants for Postdoctoral Studies, by Ram\'on Areces Foundation.}

	

%%% ----------------------------------------------------------------------

\begin{abstract}
A linear operator $U$ acting boundedly on an infinite-dimensional separable complex Hilbert space $H$ is universal if every linear bounded operator acting on $H$ is similar to a scalar multiple of a restriction of $U$ to one of its invariant subspaces. It turns out that characterizing the lattice of closed invariant subspaces of a universal operator is equivalent to solve the Invariant Subspace Problem for Hilbert spaces. In this paper, we consider invertible weighted hyperbolic composition operators and we prove the universality of the translations by eigenvalues of such operators, acting on Hardy and weighted Bergman spaces. Some consequences for the Banach space case are also discussed.
	
\end{abstract}

%%% ----------------------------------------------------------------------
\maketitle
%%% ----------------------------------------------------------------------

\section{Introduction and preliminaries}

The \textit{Invariant Subspace Problem} is, probably, one of the most important open problems in Operator Theory on Hilbert spaces. The problem asks if every linear operator acting boundedly on a infinite-dimensional separable complex Hilbert space has a non-trivial closed invariant subspace. Despite its apparent simplicity and the development of various techniques to approach the problem during the last decades (see, for instance, \cite{CP,RR}), the problem is still unsolved.

\medskip

One of the most promising approaches arising from those techniques consists of studying universal operators. If $H$ is an infinite-dimensional separable complex Hilbert space and  $\EL(H)$ denotes the algebra of linear bounded operators acting on $H$, an operator $U \in \EL(H)$ is \textbf{universal} if for every operator $T\in \EL(H)$ there exist $\lambda \in \C$, $M \in \Lat U$ (that is, a closed invariant subspace for $U$) and an isomorphism $\Phi : H \rightarrow M$ such that 
\begin{equation}\label{definicion universal}
	\Phi T = \lambda U\mid_M \Phi.
\end{equation}
 In other words, $T$ is similar to $\lambda U\mid_M.$

\medskip

Universal operators were first introduced by Rota \cite{Rota1,Rota2}, who also showed that the backward shift operator (of infinite multiplicity) acting on the space $\ell_2(H)$ is, indeed, a universal operator. As an immediate consequence of the definition of universal operators the next result follows:

\begin{theorem}\label{aplicacion universales}
Let $U \in \EL(H)$ be a universal operator. Then, the following conditions are equivalent:
\begin{enumerate}
    \item [(i)] Every operator $T \in \EL(H)$ has a non-trivial closed invariant subspace.
    \item [(ii)] Every minimal non-trivial closed invariant subspace of $U$ has dimension one.
\end{enumerate}
\end{theorem}
Note that a non-trivial closed invariant subspace for $U$ is said to be \textit{minimal} if it contains no proper non-trivial closed invariant subspaces for $U$. As a consequence of the previous result, to characterize explicitly the lattice of closed invariant subspaces of a universal operator acting on $H$ is equivalent to solve the Invariant Subspace Problem for Hilbert spaces. A common strategy following this path is to produce `treatable' examples of universal operators in order to work with concrete universal operators to study their lattice of closed invariant subspaces. Next result, which was proved by Caradus \cite{Caradus}, provides a systematic tool to produce universal operators.

\begin{theorem}\label{Caradus} Let $T \in \EL(H)$ and assume that the following conditions are satisfied:
\begin{enumerate}
    \item [(i)] $\ker T$ has infinite dimension.
    \item [(ii)] $T$ is surjective.
\end{enumerate}
Then, $T$ is universal.
\end{theorem}

\subsection{Universality of (weighted) composition operators}
For a holomorphic mapping $\phi: \DD \to \DD$ (where $\DD := \{z \in \CC \, : \, |z|<1\}$) let $C_\phi$ denote the \textit{composition operator} with symbol $\phi$ given by $C_\phi f := f\circ \phi$,  which acts on spaces of holomorphic functions on $\DD$. Let $\psi$ be an \textit{hyperbolic automorphism}, i.e., a biholomorphic function $\psi: \D \to \D$ which has no fixed point on $\DD$ and its continuous extension to $\overline \D$ has exactly two distinct fixed points (on $\partial \DD$), see Subsection \ref{NotationSubSect} for more details.

Making use of Caradus' result, Nordgren, Rosenthal and Wintrobe showed in \cite{NRW} that the operators $C_\psi - \lambda I$ are universal acting on the Hardy space $H^2$ for such an hyperbolic automorphism $\psi$, where $\lambda$ is an eigenvalue of $C_\psi.$  In particular, Nordgren et al. prove that $\ker C_{\psi}-\lambda I$ is infinite-dimensional by taking products between an eigenvector of eigenvalue $\lambda$ and powers of a Blaschke product.  The methods in \cite{NRW} to prove the surjectivity of the operators $C_{\psi}-\lambda I$ relies in representing the unicellular compression of the operator to one of its closed reducing subspaces as a weighted bilateral shift acting on $\ell_2,$ and such construction does not seem exportable for more general operators, or different spaces. Note that the operators $C_{\psi}$ and $(C_{\psi}-\lambda I)$ share their lattices of closed invariant subspaces, so it is enough to study the lattice of the operator $C_{\psi}$ in order to apply Theorem \ref{aplicacion universales}. The cyclic vectors for hyperbolic composition operators have been widely studied in the recent years (see, for instance, \cite{GG, Matache1, Matache2} and the references therein).
\medskip

Adjoints of certain Toeplitz operators acting on the Hardy and the Bergman space have also been proved to be universal \cite{CG3, CG4}. Even more, Cowen and Gallardo-Gutiérrez provide in \cite{CG2} an alternative (and simpler) proof of the universality of $C_\psi - \lambda I$ than the one given in \cite{NRW}. In particular they prove the surjectiveness of $C_\psi- \lambda I$ by a similarity result between the group $(C_{\psi_t})_{t \in \R}$ (where $(\psi_t)_{t\in \RR}$ is a group of hyperbolic automorphisms with $\psi_1 = \psi$) and a group $(T_{\phi_t}^*)_{t \in \R}$ of adjoints of Toeplitz operators acting on $H^2$. On the other hand, an equivalent result was recently proved by Carmo and Noor \cite{CN} for the hyperbolic non-automorphism case. \medskip

A natural generalization of composition operators is given by \textit{weighted composition operators}, i.e., operators $u C_\phi f := u \cdot (f\circ\phi)$ with holomorphic functions $u: \D \to \CC$ and $\phi: \D \to \D$, and where $\cdot$ stands for the pointwise product of two functions. Weighted composition operators are fundamental objects in mathematical analysis and arise naturally in many situations. For instance, all surjective isometries of the Hardy space $H^p$ and of the weighted Bergman spaces $\mathcal{A}_{\sigma}^p,$ with $p\neq 2,$  are weighted composition operators, see  \cite{forelli1964isometries, kolaski}. Weighted composition operators also appear in the study adjoints of (unweighted) composition operators and in the study of commutants of multiplication operators, and they play a role in the theory of $C_0$-semigroups, just to mention a few examples \cite{cowen1988, konig90}. \medskip

In view of the above, it seems natural to study universality properties of the translations of weighted composition operators acting on Hilbert spaces of holomorphic functions. In resemblance to the unweighted case, one may start such a study by considering invertible operators $uC_\psi$ with hyperbolic symbol $\psi$. This task is far from trivial since, opposite to the unweighted case, the spectrum of these operators is only known under certain assumptions on $u$. For instance, when considering Hardy and weighted Bergman spaces, the spectrum of such an operator $uC_\psi$ is only known if $u$ has continuous extension to $\overline \D$ and the absolute value of the evaluations of $u$ at the fixed points of $\psi$ satisfy certain inequalities \cite{gunatillake, hyvarinen2013spectra}. Its spectrum is also known if the operator $uC_\psi$ can be embedded into a $C_0$-group of weighted composition operators \cite{AGMO}. 
\medskip

In the main result of this paper (Theorem \ref{universal}), we give a sufficient condition on $u$ to obtain the universality of the translations $uC_\psi - \lambda I$ for suitable $\lambda \in \CC$. This sufficient condition, which is an inequality involving $\liminf_{z\to a} |u(z)|$ and $\limsup_{z\to b} |u(z)|$ where $a,b$ are the fixed points of $\psi$, is very explicit and easy to check on a given $u$, which turns Theorem \ref{universal} into a powerful tool to build a large family of universal operators. In particular, in the unweighted case $u=1$ our result extends the already known results in the Hardy space $H^2$ \cite{CG2, NRW} to the weighted Bergman spaces $\mathcal A_\sigma^2$. Even more, our techniques do not rely on Hilbert space methods and are easily adaptable to the Banach space setting for Hardy spaces $H^p$ and weighted Bergman spaces $\mathcal A_\sigma^p$, as we point out in Section  \ref{BanachSpaceSection}. On the other hand, in order to prove the universality of $uC_\psi - \lambda I$ we prove new results on the spectral sets of $uC_\psi$ which seem interesting by themselves and are meaningful contributions in the understanding of weighted composition operators \cite{gunatillake, hyvarinen2013spectra}.

\subsection{Hilbert spaces of analytic functions on $\D$}

Let us denote by $\mathcal{O}(\D)$ the Fr\'echet space of all holomorphic functions on the disc. Along the paper, we work with Hardy and weighted Bergman spaces on $\D.$ We recall the definitions of these Hilbert subspaces of $\mathcal O(\DD)$. 

\medskip

We denote by $H^2$ the Hardy space on $\D$ formed by all functions $f\in \mathcal{O}(\D)$ such that
$$
\|f\|_{H^2} := \sup_{0<r<1} \left(\int_0^{2\pi} |f(re^{i\theta})|^2 \,\frac{d\theta}{2\pi} \right)^{1/2} < \infty,
$$
endowed with the norm $\Vert\cdot\Vert_{H^2}$.

\medskip

Let $\sigma>-1.$ The weighted Bergman space, denoted by $\mathcal{A}_\sigma^2,$ is formed by all holomorphic functions in $\D$ such that  
$$
\|f\|_{\mathcal{A}_\sigma^2} := \left(\int_{\D} |f(z))|^2 (1-|z|^2)^\sigma dA(z)\right)^{1/2} < \infty,
$$ 
where $dA$ is the Lebesgue measure on $\D$.

\medskip

For every $\phi\in Aut(\D)$ and $\gamma\geq 0,$ let us denote $C_{\phi,\gamma}$ the weighted composition operator  given by $C_{\phi,\gamma}:=(\phi')^{\gamma}C_{\phi},$  where $\phi'$ is the derivative of $\phi.$ These operators are isometries in $H^2$ and in $\mathcal{A}_\sigma^2$ for $\gamma=1/2$ and $\gamma=(\sigma+2)/2$, see \cite[Theorem 2]{forelli1964isometries} and \cite[Theorem 4.6]{hyvarinen2013spectra}.

\subsection{Notation}\label{NotationSubSect}

We set some notation for the remainder of the paper:

\begin{enumerate}
	\item $H$ denotes indistinctly the Hardy or weighted Bergman space, and $\gamma$ the corresponding parameter for which the operators $C_{\phi,\gamma}$ are isometries on $H$ for every $\phi$ automorphism of $\D.$ Also, $\|\cdot\|_H$ denotes the corresponding space norm introduced above, and $\|\cdot\|_{\mathcal{L}(H)}$ the operator norm.
	\item $\psi: \DD \to \DD$ is a hyperbolic automorphism of $\DD$, i.e., a biholomorphic mapping from $\DD$ onto $\DD$ (hence a Möbius transform) with no fixed points on $\DD$ and exactly two fixed points on $\partial \DD$. Let  $a$ (attractive) and $b$ (repulsive) $\in \partial \DD$ be such fixed points, that is, $\lim_{n\to\infty} \psi_n(z)=a$ and $\lim_{n\to-\infty} \psi_n(z)=b$ for each $z\in \D$ where $\psi_n := \underbrace{\psi \circ \ldots \circ \psi}_{n \text{ times}}$ and $\psi_{-n}:=(\psi^{-1})_n$ for $n \in \NN$. 
	
	It is well known that  $\psi$ is holomorphic in a disc (centered at the origin) of radius strictly greater than $1$ and $\psi'(a) \in (0,1), \psi'(b)\in (1,\infty)$ with $\psi'(a) \psi'(b) = 1$. %, see for instance \cite[Section 1.8]{bracci2020continuous}.
	Moreover, one has 
	\begin{equation}\label{automorphismEq}
		\psi(z) = \frac{(b \psi'(a)-a)z +ab(1-\psi'(a))}{(\psi'(a)-1)z +b-a\psi'(a)}, \qquad z \in \DD. 
	\end{equation}
	Also, there exists a unique $r \in (0,1)$ for which $\psi$ is conjugate to the autormorphism $z \mapsto \frac{r+z}{1+rz}$ with fixed points $-1$ and $1$.
	%see for instance \cite[Corollary 8.2.7]{bracci2020continuous}.
	We refer the reader to \cite[Sections 1.2 \& 8.2]{bracci2020continuous} for these and more details about hyperbolic autormorphisms of $\DD$.
	%Also, $\psi^{-1}$ is the inverse hyperbolic automorphism, with fixed points $b$ (attractive) and $a$ (repulsive). In this case, $(\psi^{-1})'(a)=1/\psi'(a)$ and $(\psi^{-1})'(b)=1/\psi'(b).$
	\item $u$ is a bounded holomorphic function on $\DD$ which is bounded away from zero, i.e., $\inf_{z\in \DD}|u(z)|>0$. Moreover, we set
	\begin{align*}
		A^+ &:= \limsup_{\DD \ni z\to a} |u(z)|, \qquad A^-:= \liminf_{\DD \ni z\to a} |u(z)|, \\
		B^+ &:=\limsup_{\DD \ni z\to b} |u(z)|, \qquad B^-:= \liminf_{\DD \ni z\to b} |u(z)|.
	\end{align*}
\end{enumerate} 
It is known that item (3) above is an equivalent condition to the boundedness and invertibility of $u C_{\psi}$ on $H$, see \cite[Theorem 2.0.1]{gunatillake} and \cite[Corollary 3.7]{hyvarinen2013spectra}.
%Then, $u C_{\psi}$ is an invertible weighted composition operator acting on $H$ given by $$uC_{\psi}f(z):=u(z)f(\psi(z)),\quad z\in\D,\,f\in  H.$$
% In both cases (Hardy and weighted Bergman \textcolor{blue}{spaces}), the inverse operator of $u C_{\psi}$ is the weighted composition operator given by $\displaystyle{\frac{1}{u \circ \psi^{-1}} C_{\psi^{-1}}}.$
\medskip

Also, through the rest of the paper $\|\cdot\|_{\infty}$ is the supremum norm for bounded holomorphic functions on $\D$, and $r(T)_{\mathcal{L}(H)}$ denotes the spectral radius of a bounded operator $T \in \mathcal L(H).$

\subsection{Main results} The main result of the paper is the following one, which is a direct consequence of Caradus' Theorem (Theorem \ref{Caradus}) once we have proven that such operators $uC_{\psi}-\lambda I$ satisfy the hypothesis of the aforementioned result, see Section \ref{SupraInjec} (Theorem \ref{Supra} and Theorem \ref{Injec}).

\begin{theorem}\label{universal}
Every invertible  weighted composition operator $uC_{\psi}$ with hyperbolic symbol acting on $H$ such that $\frac{B^+}{(\psi'(b))^{\gamma}}< \frac{A^-}{(\psi'(a))^{\gamma}},$ satisfies that $uC_{\psi}-\lambda I$ is universal for each $\lambda\in\C$ such that $\frac{B^+}{(\psi'(b))^{\gamma}}< |\lambda|<\frac{A^-}{(\psi'(a))^{\gamma}}.$
\end{theorem}

As a consequence, letting $u=1$ on the previous result, we obtain an extension of Nordgren, Rosenthal and Wintrobe Theorem to a wider collection of spaces. Note that in the unweighted case ($u=1$) $\sigma(C_\psi)=\{\lambda\in\C\,:\, \frac{1}{(\psi'(a))^{\gamma}}\leq |\lambda|\leq \frac{1}{(\psi'(b))^{\gamma}} \},$ see for instance \cite{hyvarinen2013spectra}.
\begin{corollary}
	If $\psi : \D\rightarrow \D$ is a hyperbolic automorphism, then the operator $C_\psi -\lambda I$ acting on $H$ is universal for every $\lambda$ in the interior of $\sigma(C_\psi).$
\end{corollary}
Since $uC_{\psi}$ and $uC_{\psi}-\lambda I$ share their lattice of invariant subspaces, it is a direct consequence of Theorem \ref{aplicacion universales} and Theorem \ref{universal} the following result.

\begin{theorem}\label{teorema subespacios}
Every linear bounded operator acting on $H$ has a non-trivial closed invariant subspace if and only if every minimal non-trivial closed invariant subspace of an invertible weighted composition operator $uC_{\psi}$ with hyperbolic symbol such that $\frac{B^+}{(\psi'(b))^{\gamma}}< \frac{A^-}{(\psi'(a))^{\gamma}}$ has dimension one.
\end{theorem}
Observe that, given an attractive and a repulsive fixed point $a,b \in \T$ and $u:\D\rightarrow \C$ a bounded holomorphic function which is bounded away from zero, there always exists a hyperbolic automorphism $\psi$ such that $\frac{B^+}{(\psi'(b))^{\gamma}}< \frac{A^-}{(\psi'(a))^{\gamma}}$ due to \eqref{automorphismEq}. This shows that the previous result can be applied for any bounded symbol $u$ inducing an invertible weighted composition operator $uC_\psi$, for a concrete hyperbolic automorphism $\psi.$  
\subsection{Summary} 

The paper is organized as follows. In Section \ref{section 2} we present Hilbert range spaces of multiplication operators, defined by the fixed points of the hyperbolic automorphism that we are considering. In particular, we state two technical lemmas which are key points in the proofs of the main results in both Section \ref{section 3} and Section \ref{SupraInjec}.  Section \ref{section 3} shows spectral inclusions for invertible weighted composition hyperbolic operators. Such results allow to improve some known spectral considerations in the literature, as we will remark in the last section. In Section \ref{SupraInjec} we prove the sufficient conditions that we need to apply Caradus' theorem in order to get the main results presented in this introduction. Finally, in Section \ref{BanachSpaceSection}, we discuss about some Banach space considerations related to the results along the paper. In fact, the techniques developed in the paper can be considered within the Banach space field.

\section{Range spaces}\label{section 2}

For $\mu,\nu\geq 0$, set $\omega_{\mu,\nu}(z) := (a-z)^{\mu} (b-z)^\nu$, $z \in \DD$. Note that  $\omega_{\mu,\nu}:\D\to \C$ are non-vanishing holomorphic functions satisfying next properties: \begin{enumerate}
		\item $|\omega_{\mu,\nu}|$ is bounded,
		\item $\displaystyle{\left|\frac{\omega_{\mu,\nu} \circ \psi}{\omega_{\mu,\nu}}\right|}$ is bounded and bounded away from zero.
	\end{enumerate}
The first one is clear. For the second one, note that $\displaystyle{\left|\frac{\omega_{\mu,\nu}\circ \psi(z)}{\omega_{\mu,\nu}(z)}\right|}$ is clearly bounded and bounded away from zero if $z$ is bounded away from $a,b$. Moreover,  $\frac{\omega_{\mu,\nu} \circ \psi}{\omega_{\mu,\nu}}$ is holomorphic at $a,b$ since $\psi-a$ and $\psi-b$ have zeroes of order $1$ at $a,b$ respectively (and $\psi$ is holomorphic in a disc of radius strictly greater than $1$). Moreover, %by the Julia-Wolff-Carathéodory theorem (see for example \cite[Theorem 1.7.3]{bracci2020continuous}), 
one has  
\begin{align}\label{limitsEq}
	\lim_{\DD \ni z \to a} \frac{\omega_{\mu,\nu}\circ \psi(z)}{\omega_{\mu,\nu}(z)} = (\psi'(a))^\mu, \qquad \lim_{\DD \ni z \to b} \frac{\omega_{\mu,\nu}\circ \psi(z)}{\omega_{\mu,\nu}(z)} = (\psi'(b))^\nu.
\end{align}
Alternatively, one can check these statements with a few computations using \eqref{automorphismEq}.	

For $\mu,\nu \geq 0$, we set
$$H_{\mu,\nu}:= \{\omega_{\mu,\nu} f \mid f \in H\}.
$$
%\textcolor{blue}{Recall that $H$ may be the Hardy space $H^2$ or a weighted Bergman space $\mathcal A_\sigma^2$.}
Therefore, $H_{\mu,\nu}$ is the range space of the injective bounded  operator $M_{\omega_{\mu,\nu} }(f) := \omega_{\mu,\nu} f$ acting on $H$. Thus, $H_{\mu,\nu}$ is a Banach space endowed with the norm
	$$\|f\|_{H_{\mu,\nu}} := \left\| \frac{f}{\omega_{\mu,\nu}}\right\|_H, \quad f \in H_{\mu,\nu}.
	$$
	With such a norm, $M_{\omega_{\mu,\nu}} : H \to H_{\mu,\nu}$ is an isometric isomorphism, and the inclusion $H_{\mu,\nu}  \hookrightarrow H$ is continuous (recall that $\omega_{\mu,\nu}$ is a multiplier of $H$).
	
We need the following key lemmas to get the main results.
	
	\begin{lemma}\label{isometricCopy}
		Let $\mu,\nu \geq 0$. The weighted composition operator $uC_\psi$ is a bounded invertible operator acting on $H_{\mu,\nu}$. Moreover, it is isometrically equivalent to the operator $\frac{\omega_{\mu,\nu} \circ \psi}{\omega_{\mu,\nu}} u C_\psi$ acting on $H$.
	\end{lemma}
	\begin{proof}
		Note that
		\begin{align*}
			u C_\psi (M_{\omega_{\mu,\nu}} f) = \omega_{\mu,\nu} \frac{\omega_{\mu,\nu} \circ \psi}{\omega_{\mu,\nu}} u C_\psi f = M_{\omega_{\mu,\nu}} \left( \frac{\omega_{\mu,\nu} \circ \psi}{\omega_{\mu,\nu}} u C_\psi f\right), \quad f \in H.
		\end{align*}
		That is, $u C_\psi = M_{\omega_{\mu,\nu}} \frac{\omega_{\mu,\nu} \circ \psi}{\omega_{\mu,\nu}} u C_\psi (M_{\omega_{\mu,\nu}})^{-1}$. As $M_{\omega_{\mu,\nu}}: H \to H_{\mu,\nu}$ is an isometric isomorphism, the claim follows.
	\end{proof}

	\begin{lemma}\label{H2decomposition}
		Let $\mu,\nu \geq 0$. Then $H = H_{\mu,0} + H_{0,\nu}$.
	\end{lemma}
	\begin{proof}
		Take $m,n\in \NN_0$ such that $m\geq \mu$ and $n \geq \nu$. Let $f \in H$. Then
		\begin{align*}
			f(z) &= \frac{(a-z - (b-z))^{m+n}}{(a-b)^{m+n}} f(z) = \frac{1}{(a-b)^{m+n}} \sum_{j=0}^{m+n} {{m+n}\choose{j}}   (-1)^{m+n-j} \omega_{j,m+n-j}(z) f(z), \quad z \in \DD.
		\end{align*}
		It is clear that $\omega_{j,m+n-j} f$ belongs to $H_{0,\nu}$ if $j \leq m$, just as such a function belongs to $H_{\mu,0}$ if $j \geq m$. Thus, the proof is done.
	\end{proof}

\section{Spectral estimates}\label{section 3}

For $n \in \NN_0$, recall that $\psi_n = \underbrace{\psi \circ \ldots \circ \psi}_{n \text{ times}}$. We set $u_n(z) := \prod_{j=0}^{n-1} u \circ \psi_j$, so one gets
	\begin{equation*}
		(u C_\psi)^n = u_n C_{\psi_n}, \qquad n \in \NN_0.
	\end{equation*}

The following result is inspired by \cite[Lemma 4.4]{hyvarinen2013spectra}.
		
	\begin{lemma}\label{boundIterationWeight}
		One has
		\begin{align*}
			\lim_{n \to \infty} \left(\sup_{z\in \DD} |u_n(z)|\right)^{1/n} & \leq \max\{A^+, B^+\},
			\\ \lim_{n \to \infty} \left(\inf_{z\in \DD} |u_n(z)|\right)^{1/n} & \geq \min\{A^-, B^-\}.
		\end{align*}
	\end{lemma}
	\begin{proof}
		Let us prove the identity regarding the supremum. For each $\varepsilon >0$, take neighborhoods $U,V$ of $a,b$ (respectively) in $\DD$ such that
		$$|u(z)| \leq (1+\varepsilon) \max\{A^+, B^+\}, \qquad z \in U \cup V.
		$$
		Note that there exists $m \in \NN_0$ such that, for all $z \in \DD$, at most $m$ elements of $\{\psi_n(z) \mid n \in \NN_0\}$ belong to $\DD \setminus \{U \cup V\}$. Thus
		\begin{align*}
			\sup_{z\in \DD} |u_n(z)| & \leq \|u\|_\infty^m \left[(1+\varepsilon) \max\{A^+, B^+\}\right]^{n-m}.
		\end{align*}
		Hence $\lim_{n \to\infty} \left(\sup_{z \in \DD} |u_n(z)|\right)^{1/n} \leq \max\{A^+, B^+\}$, as claimed (previous limit exists because $n \mapsto \|u_n\|_{\infty}$ is a subexponential function, see for instance \cite[Theorem 7.6.5]{hille1957functional}).
		
		Regarding the infimum, the proof is analogous to the one of the supremum.
	\end{proof}
	
The proof of next result follows the ideas of the proof of \cite[Theorem 4.6]{hyvarinen2013spectra}.

	\begin{proposition}\label{spectralRadiusProp}
		One has
		$$\sigma(uC_{\psi}) \subseteq \left\{\lambda \in \CC \, : \,  \min\left\{\frac{A^-}{(\psi'(a))^{\gamma}}, \frac{B^-}{(\psi'(b))^{\gamma}}\right\} \leq |\lambda| \leq  \max\left\{\frac{A^+}{(\psi'(a))^{\gamma}}, \frac{B^+}{(\psi'(b))^{\gamma}}\right\}\right\}.
		$$
	\end{proposition}
	\begin{proof}
		Since  $C_{\psi,\gamma}$ is an isometric isomorphism on $H,$
		\begin{align*}
			\|(u C_{\psi})^n f\|_H & = \|u_n C_{\psi_n} f\|_H   \leq \left\| \frac{u_n} {(\psi_n')^{\gamma}} \right\|_{\infty} \left\|(\psi_n')^{\gamma} C_{\psi_n} f\right\|_H
			= \left\| \left(\frac{u} {(\psi')^{\gamma}}\right)_n \right\|_{\infty} \|f\|_H, \quad f \in H, \, n \in \NN_0,
		\end{align*}
		where we have used $(\psi_n)' = (\psi')_n$, so $\displaystyle{\frac{u_n} {((\psi_n)')^{\gamma}}} = \left(\frac{u} {(\psi')^{\gamma}}\right)_n$.  Hence, an application of Lemma \ref{boundIterationWeight} to the function $u/(\psi')^{\gamma}$ (instead of the function $u$) yields
			\begin{align*}
				\lim_{n\to \infty} \left\|(u C_\psi)^n \right\|_{\mathcal{L}(H)}^{1/n} &\leq \max\left\{ \frac{A^+}{(\psi'(a))^{\gamma}}, \frac{B^+}{(\psi'(b))^{\gamma}}\right\},
			\end{align*}
		and the upper bound for the spectral radius follows by the spectral radius formula.
		
		Regarding the lower bound on $\{|\lambda|\, : \, \lambda \in \sigma(uC_\psi) \}$, recall that 
		$$\sigma((uC_\psi)^{-1}) = (\sigma (u C_\psi))^{-1} = \{1/\lambda \, : \,\lambda \in \sigma (u C_\psi)\}.
		$$
		Since $(uC_\psi)^{-1} = \frac{1}{u \circ \psi^{-1}} C_{\psi^{-1}}$, by what we have already proven
		\begin{align*}
			\lim_{n \to \infty} \|(u C_\psi)^{-n}\|_{\mathcal{L}(H)}^n \leq \max\left\{\frac{(\psi'(a))^{\gamma}}{A^-}, \frac{(\psi'(b))^{\gamma}}{B^-}\right\},
		\end{align*}
		and our claim follows. Note that we have used above that $\limsup_{\DD \ni z \to a} \left|\frac{1}{u \circ \psi^{-1}}\right| = (A^-)^{-1}$, $(\psi^{-1})'(a) = (\psi'(a))^{-1}$, and the analogous identities evaluated at $b$.
	\end{proof}

The other spectral inclusion that we need is the following one:

\begin{proposition}\label{inclusionSpec} Assume $\frac{B^+}{(\psi'(b))^{\gamma}}< \frac{A^-}{(\psi'(a))^{\gamma}}.$ Then $$\left\{\lambda \in \CC \, : \, \frac{B^+}{(\psi'(b))^{\gamma}} \leq   |\lambda| \leq  \frac{A^-}{(\psi'(a))^{\gamma}} \right\} \subseteq \sigma(uC_{\psi}) .$$

\end{proposition}
\begin{proof}
Since $\psi'(a)\in (0,1)$ and $\psi'(b)>1,$ there are $\alpha,\beta>0$ such that $A^+(\psi'(a))^{\alpha-\gamma}<\frac{B^+}{(\psi'(b))^{\gamma}}$ and $\frac{A^-}{(\psi'(a))^{\gamma}}< B^-(\psi'(b))^{\beta-\gamma}.$ Set $f(z):=(a-z)^{\alpha-\gamma}(b-z)^{\beta-\gamma}.$ Note that $f\in H.$ Moreover,  $f\in H_{\mu,0},H_{0,\nu}$ for each $0\leq \mu<\alpha,$ $0\leq \nu<\beta.$ Fix now $0\leq \mu<\alpha$ and $0\leq \nu<\beta$ such that  $$A^+(\psi'(a))^{\mu-\gamma}=\frac{B^+}{(\psi'(b))^{\gamma}},\qquad \frac{A^-}{(\psi'(a))^{\gamma}}= B^-(\psi'(b))^{\nu-\gamma}.$$

Recall that the inclusion $H_{\mu,0}  \hookrightarrow H$ is continuous, so $\|(uC_{\psi})^nf\|_H\leq \|(uC_{\psi})^nf\|_{H_{\mu,0}},$ and therefore $$\limsup_n \|(uC_{\psi})^nf\|_H^{1/n}\leq \limsup_n \|(uC_{\psi})^n\|^{1/n}_{\mathcal{L}(H_{\mu,0})}=r(uC_{\psi})_{\mathcal{L}(H_{\mu,0})}.$$
By Lemma \ref{isometricCopy} the operator $uC_{\psi}$  acting on $H_{\mu,0}$ is isometrically isomorphic to the operator  $\frac{\omega_{\mu,0} \circ \psi}{\omega_{\mu,0}} u C_\psi$ acting on $H.$ In particular, their spectral radius are equal. If one observes that the weight $\frac{\omega_{\mu,0} \circ \psi}{\omega_{\mu,0}} u$ is bounded and bounded away from zero, by Proposition \ref{spectralRadiusProp} applied to $\frac{\omega_{\mu,0} \circ \psi}{\omega_{\mu,0}} u C_\psi$ follows that $$\limsup_n \|(uC_{\psi})^nf\|_H^{1/n}\leq \max\left\{  \frac{B^+}{(\psi'(b))^{\gamma}}, A^+(\psi'(a))^{\mu-\gamma}\right\}= \frac{B^+}{(\psi'(b))^{\gamma}}.$$

Similarly, the inclusion $H_{0,\nu}  \hookrightarrow H$ is continuous and $(uC_{\psi})^{-1}=\frac{1}{u\circ \psi^{-1}}C_{\psi^{-1}}$ acting on $H_{0,\nu}$ is isometrically isomorphic to  $\frac{\omega_{0,\nu} \circ \psi^{-1}}{\omega_{0,\nu}} \frac{1}{u\circ \psi^{-1}}C_{\psi^{-1}}$ acting on $H,$ with both weights $\frac{1}{u\circ \psi^{-1}}$ and $\frac{\omega_{0,\nu} \circ \psi^{-1}}{\omega_{0,\nu}} \frac{1}{u\circ \psi^{-1}}$ bounded and bounded away from zero. Hence, applying again Proposition \ref{spectralRadiusProp} we get $$\limsup_n \|(uC_{\psi})^{-n}f\|_H^{1/n}\leq \max\left\{  \frac{1/A^-}{((\psi^{-1})'(a))^{\gamma}}, \frac{((\psi^{-1})'(b))^{\nu-\gamma}}{B^-}\right\}= \frac{(\psi'(a)^{\gamma})}{A^-}.$$

For each $\lambda\in\C$ such that $|\lambda|> \frac{B^+}{(\psi'(b))^{\gamma}}$, consider the function $F_\lambda \in H$ given by the convergent series (on $H$) $$F_{\lambda}:=\sum_{n=0}^{\infty}\frac{(uC_{\psi})^n f}{\lambda^{n+1}},$$ and for $\lambda\in\C$ such that $|\lambda|< \frac{A^-}{(\psi'(a))^{\gamma}}$, consider the function $G_\lambda \in H$ given by the convergent series (on $H$) $$G_{\lambda}:=\sum_{n=0}^{\infty}\lambda^n (uC_{\psi})^{-(n+1)} f.$$ 
Indeed, note that these two series are Cauchy sequences on $H$ by the asymptotic bounds proven above for $(u C_{\psi})^n f$ and $(uC_{\psi})^{-n}f$ respectively. One can easily check that $$(\lambda -uC_{\psi})F_{\lambda}=f,\quad |\lambda|> \frac{B^+}{(\psi'(b))^{\gamma}},$$ and $$(\lambda -uC_{\psi})G_{\lambda}=f,\quad |\lambda|< \frac{A^-}{(\psi'(a))^{\gamma}}.$$

Returning to the claim of the proposition, we assume by contradiction that there is $\tilde{\lambda}\in\C$ with $\frac{B^+}{(\psi'(b))^{\gamma}}<|\tilde{\lambda}|<\frac{A^-}{(\psi'(a))^{\gamma}}$ such that $\tilde{\lambda}\in \rho(u C_{\psi})$. In such a case, since the resolvent set $\rho(u C_{\psi})$ is open, there is a small open ball $V$ inside of the annulus of radii $\frac{B^+}{(\psi'(b))^{\gamma}}$ and $\frac{A^-}{(\psi'(a))^{\gamma}}$ which also belongs to the resolvent set. Then, by above calculations, for each $\lambda\in V$ we have $$G_{\lambda}=(\lambda -uC_{\psi})^{-1}f=F_{\lambda}.$$ Since  $F_{(\cdot)}:\{\lambda\in\C\ :\  |\lambda|>\frac{B^+}{(\psi'(b))^{\gamma}}\}\to H$ and $G_{(\cdot)}:\{\lambda\in\C\ :\  |\lambda|<\frac{A^-}{(\psi'(a))^{\gamma}} \}\to H$ are vector-valued analytic functions, by the uniqueness of analytic continuation $F_{\lambda}=G_{\lambda}$ for all $\lambda$ in the annulus of radii $\frac{B^+}{(\psi'(b))^{\gamma}}$ and $\frac{A^-}{(\psi'(a))^{\gamma}}$ (see \cite[Corollary A.4]{ABHN}). 

Then one has that the vector-valued entire function (from $\C$ to $H$) given by \begin{equation}\label{function}
\lambda\mapsto \left\{\begin{array}{ll}
F_{\lambda},&|\lambda|>\frac{B^+}{(\psi'(b))^{\gamma}},\\
G_{\lambda},&|\lambda|<\frac{A^-}{(\psi'(a))^{\gamma}},
\end{array} \right.
\end{equation}
is also bounded. Indeed, take $\varepsilon>0$ small enough such that $\frac{B^+}{(\psi'(b))^{\gamma}}+\varepsilon<\frac{A^-}{(\psi'(a))^{\gamma}}-\varepsilon,$ and so $\|F_{\lambda}\|_{H}\lesssim \frac{1}{|\lambda|-\frac{B^+}{(\psi'(b))^{\gamma}}}$ is uniformly bounded for all $|\lambda|\geq \frac{B^+}{(\psi'(b))^{\gamma}}+\varepsilon$ (even more $\|F_{\lambda}\|_{H}\to 0$ as $|\lambda|\to \infty$) and $\|G_{\lambda}\|_{H}\lesssim \frac{1}{\frac{A^-}{(\psi'(a))^{\gamma}}-|\lambda|}$ is uniformly bounded for all $|\lambda|\leq  \frac{A^-}{(\psi'(a))^{\gamma}}-\varepsilon.$

By the vector-valued Liouville theorem (see for example \cite[Theorem 6]{Barletta}), one concludes that the entire function given by \eqref{function} is constant, and since 	$\|F_{\lambda}\|_{H}\to 0$ as $|\lambda|\to \infty,$ in particular is the zero function. This contradicts the identity $0\neq f=(\lambda -uC_{\psi})F_{\lambda}=(\lambda -uC_{\psi})G_{\lambda}$ for $\lambda$ in the annulus of radii $\frac{B^+}{(\psi'(b))^{\gamma}}$ and $\frac{A^-}{(\psi'(a))^{\gamma}},$ and we conclude $\widetilde \lambda \in \sigma(uC_{\psi})$ for all $\widetilde \lambda \in \CC$ with $\frac{B^+}{(\psi'(b))^{\gamma}}<|\tilde{\lambda}|<\frac{A^-}{(\psi'(a))^{\gamma}}$. Equivalently, $\left\{\lambda \in \CC \, : \, \frac{B^+}{(\psi'(b))^{\gamma}} <   |\lambda| < \frac{A^-}{(\psi'(a))^{\gamma}} \right\} \subseteq \sigma(uC_{\psi}),$ and the claim follows since $\sigma(uC_{\psi})$ is a closed subset of $\C.$
\end{proof}

\section{Hypotheses of Caradus theorem}\label{SupraInjec}

	In this section, we prove the hypotheses of Caradus' Theorem, which imply Theorem \ref{universal}.

	\begin{theorem}\label{Supra}
		Assume $\frac{B^+}{(\psi'(b))^\gamma} < \frac{A^-}{(\psi'(a))^\gamma}$. Then, for each $\lambda \in \CC$ with $\frac{B^+}{(\psi'(b))^\gamma} < |\lambda| < \frac{A^-}{(\psi'(a))^\gamma}$, the operator $\lambda I - u C_\psi$ is surjective on $H$.
	\end{theorem}
	\begin{proof}
		Let $\lambda \in \CC$ with $\frac{B^+}{(\psi'(b))^\gamma} < |\lambda| < \frac{A^-}{(\psi'(a))^\gamma}.$ Take $\mu,\nu > 0$ such that 
		\begin{equation*}
			|\lambda| >  A^+ (\psi'(a))^{\mu-\gamma}, \qquad |\lambda| < B^- (\psi'(b))^{\nu-\gamma}.
		\end{equation*}
		Note that such $\mu,\nu$ exist since $\psi'(a) \in (0,1)$ and $\psi'(b) \in (1,\infty)$. Also observe that by \eqref{limitsEq} one has
		$$\limsup_{\DD \ni z \to a } \left|\frac{\omega_{\mu,0}\circ \psi(z)}{\omega_{\mu,0}(z)} u(z)\right| =  (\psi'(a))^{\mu} A^+, \qquad	\limsup_{\DD \ni z \to b }  \left|\frac{\omega_{\mu,0}\circ \psi(z)}{\omega_{\mu,0}(z)} u(z)\right| =  B^+,
		$$
		and 
		$$\liminf_{\DD \ni z \to a} \left|\frac{\omega_{0,\nu}\circ \psi(z)}{\omega_{0,\nu}(z)} u(z)\right| =  A^-,
		\qquad
		\liminf_{\DD \ni z \to b} \left|\frac{\omega_{0,\nu}\circ \psi(z)}{\omega_{0,\nu}(z)} u(z)\right| =  (\psi'(b))^\nu B^-.
		$$
		Hence, Proposition \ref{spectralRadiusProp} yields $\displaystyle{\lambda \notin \sigma\left(\frac{\omega_{\mu,0}\circ \psi}{\omega_{\mu,0}} u C_\psi\right)}$ and $\displaystyle{\lambda \notin \sigma\left(\frac{\omega_{0,\nu}\circ \psi}{\omega_{0,\nu}} u C_\psi\right)}$, regarding both operators as operators in $H$. In particular, the operators $\lambda-\frac{\omega_{\mu,0}\circ \psi}{\omega_{\mu,0}} u C_\psi$ and $ \lambda -\frac{\omega_{0,\nu}\circ \psi}{\omega_{0,\nu}} u C_\psi$ are surjective on $H$.
		
		Now, by Lemma \ref{isometricCopy}, the operator $u C_\psi$ is a bounded invertible operator on $H_{\mu,0}$ (and also on $H_{0,\nu}$), which is isometrically equivalent to the operator $\frac{\omega_{\mu,0}\circ \psi}{\omega_{\mu,0}} u C_\psi$ ($\frac{\omega_{0,\nu}\circ \psi}{\omega_{0,\nu}} u C_\psi$ respectively) acting on $H$. By the above, we conclude that $\lambda-uC_\psi$ is surjective as an operator both on $H_{\mu,0}$ and on $H_{0,\nu}$. As $H = H_{\mu,0} + H_{0,\nu}$ by Lemma \ref{H2decomposition} (with $H_{\mu,0}, H_{0,\nu} \subseteq H$), our claim follows and the proof is done.
		
\end{proof}

\begin{remark} Note that by Proposition \ref{inclusionSpec} and Theorem \ref{Supra}, if $\frac{B^+}{\psi'(b)^\gamma} < \frac{A^-}{\psi'(a)^\gamma},$ each $\lambda \in \CC$ with $\frac{B^+}{\psi'(b)^\gamma} < |\lambda| < \frac{A^-}{\psi'(a)^\gamma}$ belongs to the point spectrum of $uC_{\psi}.$

\end{remark}

\begin{theorem}\label{Injec}
Assume $\frac{B^+}{(\psi'(b))^\gamma} < \frac{A^-}{(\psi'(a))^\gamma}$. Then  $\ker (\lambda I- uC_{\psi})$ has infinite dimension for each $\lambda \in \CC$ with $\frac{B^+}{(\psi'(b))^\gamma} < |\lambda| < \frac{A^-}{ 	(\psi'(a))^\gamma}.$ 
\end{theorem}
\begin{proof}
Let $\lambda \in \CC$ with $\frac{B^+}{(\psi'(b))^\gamma} < |\lambda| < \frac{A^-}{(\psi'(a))^\gamma}.$ By the above remark, there is a non-zero function $f\in H$ such that $\lambda f=uC_{\psi} f.$ Let $\delta:=-\log \psi'(a) \in (0,\infty),$ and set $$g_k(z):=\biggl( \frac{b-z}{a-z} \biggr)^{\frac{2\pi k i }{\delta}},\quad z\in\D,\, k\in\ZZ.$$ Note that $C_{\psi} g_k(z)=g_k(z)$ for $z\in\D,\,k\in\ZZ$. Indeed, it is readily seen from \eqref{automorphismEq} that $C_{\psi }\biggl( \frac{b-z}{a-z}\biggr)^{w}=e^{\delta w}\biggl( \frac{b-z}{a-z} \biggr)^{w}$ for $w\in\C$.

Note that each $g_k$ is a bounded holomorphic function on $\D,$ therefore $f_k:=g_k f\in H$ for all $k\in \ZZ.$ Moreover, $$uC_{\psi} f_k= (uC_\psi f) \cdot (C_\psi g_k) = \lambda f g_k = \lambda f_k, \qquad k \in \ZZ.$$
In consequence, $f_k\in \ker (\lambda I- uC_{\psi})$ for each $k\in\ZZ,$ being $f_k$ a non zero-function.

Finally, observe that $\{f_k\}_{k\in\ZZ}$ is a linearly independent family of holomorphic functions if and only $\{g_k\}_{k\in\ZZ}$ are also linearly independent. The last one assertion follows since the set of functions $\{g_k\}_{k\in\ZZ}$ are eigenfunctions with distinct eigenvalues of the linear operator $T,$ where $(Th)(z)=\omega_{1,1}(z)h'(z).$ In particular $Tg_k=\frac{2\pi k(b-a) i}{\delta}g_k.$ 

In conclusion the linear space generated by $\{f_k\}_{k\in\ZZ}$ is infinite-dimensional and then, since $f_k\in \ker (\lambda I- uC_{\psi})$ for each $k\in\ZZ,$ the result follows.

\end{proof}

 \subsection*{A remark on non-hyperbolic weighted composition operators}
	It is natural to ask whether a result of the kind of Theorem \ref{universal} may be stated for invertible weighted composition operators $uC_\psi$ acting on $H$ if $\psi$ is not a hyperbolic automorphism. 
	
In such a case, recall that $\psi$ has to be either a parabolic automorphism, fixing then only one point which lies on $\T$, or either $\psi$ is an elliptic automorphism, fixing a point in $\D$ and another in $\C_\infty\setminus \overline{\D},$ where $\C_\infty := \C\cup\{\infty\}$ denotes the Riemann sphere. 
	
Now, by \cite[Theorem 2.2]{ST}, if $uC_\psi - \lambda I$ is universal for some $\lambda \in \C$, then the point spectrum of $uC_\psi$ has non-empty interior. But if $u$ is in the disc algebra, by \cite[Theorems 4.3, 4.11 and 4.13]{hyvarinen2013spectra} the spectrum of $uC_\psi$ has empty interior if $\psi$ is a parabolic or an elliptic automorphism, so $uC_\psi-\lambda I$ cannot be universal for any $\lambda \in \C.$
	
It would be interesting to characterize whether the spectrum of $uC_\psi$ has non-empty interior for symbols $u:\D\rightarrow \C$ not lying in the disc algebra, which would then open the door to study the universality for the non-hyperbolic cases.

\section{Banach space considerations}\label{BanachSpaceSection}
The concept of universal operator that we have treated throughout the text has been defined only for linear operators acting boundedly on Hilbert spaces. However, one may consider the same concept for operators acting on Banach spaces, with an equivalent definition to the one given in \eqref{definicion universal}.
	
The main obstacle to work with universality on Banach spaces is that not every infinite-dimensional closed invariant subspace of an operator $T \in \EL(\X)$ is isomorphic to $\X$, as it happens in the Hilbert space setting.   
	
Nevertheless, Caradus' Theorem may be re-written to obtain a sufficient condition for a linear bounded operator acting on a Banach space $\X$ to be universal \cite{Caradus}:
	
	\begin{theorem}\label{caradus Banach}Let $\X$ be a infinite-dimensional separable complex Banach space and $T\in \EL(\X)$, and assume that the following conditions are satisfied:
		\begin{enumerate}
			\item [(i)] $\ker T$ is a complemented subspace that contains a subspace which is isomorphic to $\X.$
			\item [(ii)] $T$ is surjective.
			\end{enumerate}
		Then, $T$ is universal.
	\end{theorem}
This result may be used to show, for instance, that the backward shift acting on $\ell_1(\ell_1)$ is a universal operator.

Following this line of ideas, one may consider the Banach space version of the Hardy spaces and the weighted Bergman spaces. Namely, the Hardy space $H^p$ ($1\leq p < \infty$) on $\D$ is the space of all functions $f \in \mathcal{O}(\D)$ such that $$\|f\|_{H^p} := \sup_{0<r<1} \left(\int_0^{2\pi} |f(re^{i\theta})|^p \,\frac{d\theta}{2\pi} \right)^{1/p} < \infty,
$$
endowed with the norm $\Vert\cdot\Vert_{H^p}$.
	
	\medskip

Equivalently, for any $\sigma>-1,$ the weighted Bergman space $\mathcal{A}_\sigma^p$ ($1\leq p < \infty$) is formed by all holomorphic functions in $\D$ such that  
$$
\|f\|_{\mathcal{A}_\sigma^p} := \left(\int_{\D} |f(z))|^p (1-|z|^2)^\sigma dA(z)\right)^{1/p} < \infty,
$$ 
where $dA$ is the Lebesgue measure on $\D$.
	
	\medskip
	
Let us denote indistinctly by $\mathcal{B}$ the Hardy spaces $H^p$ and the weighted Bergman spaces $\mathcal{A}^p_\sigma.$ For these spaces, let $\gamma := 1/p$ and $\gamma:= (\sigma+2)/p$ respectively, so $(\phi')^\gamma C_{\phi}$ is an isometric isomorphism on $\mathcal B$, see \cite{forelli1964isometries, kolaski}. Then, one may repeat the arguments exposed throughout the paper to deduce the following result:
	
	\begin{theorem}\label{teorema Banach}
Let $uC_\psi$ be an invertible weighted composition operator with hyperbolic symbol $\psi$ acting on $\mathcal{B}$ such that $\frac{B^+}{(\psi'(b))^{\gamma}}< \frac{A^-}{(\psi'(a))^{\gamma}}.$ Then, for every $\frac{B^+}{(\psi'(b))^{\gamma}}< |\lambda|< \frac{A^-}{(\psi'(a))^{\gamma}},$ $\ker(uC_\psi-\lambda I)$ is infinite dimensional and $uC_\psi-\lambda I$ is surjective.
	\end{theorem}
	
	Observe that this result is not sufficient to obtain the universality for the operators $uC_\psi-\lambda I$ acting on the Banach spaces $\B.$ Indeed, by Theorem \ref{caradus Banach}, it would be enough to answer the following question, that remains open, at least for the knowledge of the authors:
	
	\medskip
	
	\noindent \textbf{Question. }  Let $uC_\psi$ be an invertible weighted composition operator with hyperbolic symbol acting on $\mathcal{B}$, where $\mathcal{B}$ denotes the Hardy space $H^p$ or the weighted Bergman space $\mathcal{A}_\sigma^p,$ where $1\leq p < \infty$ and $\sigma >-1.$ Assume that $\frac{B^+}{(\psi'(b))^{\gamma}}< \frac{A^-}{(\psi'(a))^{\gamma}}$ and let $\frac{B^+}{(\psi'(b))^{\gamma}}< |\lambda|< \frac{A^-}{(\psi'(a))^{\gamma}}.$ Is $\ker(uC_\psi-\lambda I)$ a complemented subspace that contains a subspace which is isomorphic to $\B$? 
	
	\medskip
	
We also observe that the analogous results of Proposition \ref{spectralRadiusProp}, Proposition \ref{inclusionSpec} and Theorem \ref{Injec} adapted to $H^p$ and $\mathcal A_\sigma^p$ considerably improve previous results in the literature. Indeed, the spectrum of $uC_{\psi}$ was given in \cite[Theorem 4.9]{hyvarinen2013spectra} provided that the function $u$ belongs to the disc algebra (i.e. $u$ has continuous extension to $\overline \DD,$ and then we write $A:=A^+=A^-$ and $B:=B^+=B^-$) and  the inequality $\frac{B}{(\psi'(b))^{\gamma}}\leq \frac{A}{(\psi'(a))^{\gamma}}$ holds. They also proved in \cite[Remark 4.10]{hyvarinen2013spectra} (mimicking the ideas of \cite[Subsection 3.5]{gunatillake}) that every interior point of $\sigma(uC_\psi)$ is an eigenvalue of infinite multiplicity for $uC_\varphi$ if in addition one assumes that $u'$ is bounded and that $1 < \left|\frac{u(a)}{u(b)}\right| < \varphi'(a)^{-2\gamma}$. So, under the weaker assumption that the modulus of $u$, $|u|$, can be extended continuously to the fixed points of $\psi$ (and in this case we also set $A:=A^+=A^-$ and $B:=B^+ =B^-$) we have the next improvement:

\begin{corollary}\label{Spec} Let $uC_\psi$ be an invertible weighted composition operator with hyperbolic symbol $\psi$ acting on $\mathcal{B}$. Assume further that the modulus of $u$ has continuous extension to the fixed points $a$ and $b,$ and $\frac{B}{(\psi'(b))^{\gamma}}\leq \frac{A}{(\psi'(a))^{\gamma}}$. Then $$\sigma(uC_{\psi}) =\left\{\lambda \in \CC \, : \, \frac{B}{(\psi'(b))^{\gamma}} \leq   |\lambda| \leq  \frac{A}{(\psi'(a))^{\gamma}} \right\}.$$ Moreover, for every $\lambda$ belonging to the interior of $\sigma(uC_\psi)$, $\ker(uC_\psi-\lambda I)$ is infinite dimensional and $uC_\psi -\lambda I$ is surjective.
\end{corollary}
\begin{proof} If $\frac{B}{(\psi'(b))^{\gamma}}< \frac{A}{(\psi'(a))^{\gamma}}$ the result follows directly by Proposition \ref{spectralRadiusProp}, Proposition \ref{inclusionSpec} and Theorem \ref{Injec}. The case $\frac{B}{(\psi'(b))^{\gamma}}= \frac{A}{(\psi'(a))^{\gamma}}$ is analogous to the proof of \cite[Theorem 4.8]{hyvarinen2013spectra}.

\end{proof}

	\medskip
	
	Finally, to end up this section, let us point out that the construction carried out throughout Sections \ref{section 2}, \ref{section 3} and \ref{SupraInjec} can be adapted to cover a wide list of classical Banach spaces of analytic functions. We take this viewpoint, as well as the study of the spectral sets of $uC_\psi$ for a general $u$, in a forthcoming paper.
	
\section*{Conflict of interest statement}
The authors state that there is no conflict of interest.

\section*{Data availability}
No data was used for the research described in the article.

\end{document}